\documentclass[11pt]{article}  

\newcommand{\la}{\lambda}

\newcommand{\var}{\varphi}

\newcommand{\Del}{\Delta}
\newcommand{\til}{\tilde}

\newcommand{\ga}{\gamma}

\newcommand{\lo}{\longrightarrow}
\newcommand{\Lo}{\Longrightarrow}
\newcommand{\LO}{\Longleftrightarrow}

\newcommand{\sig}{\sigma}

\newcommand{\BC}{\mathbb{C}}

    \usepackage{graphicx}
    \usepackage{amsmath}
    \usepackage{amsthm}
    \usepackage{amssymb}
    \usepackage{latexsym}
    \usepackage{amsmath, mathtools, amssymb, amsthm, amsfonts}
\usepackage[numbers,comma,sort&compress]{natbib}
\usepackage{graphicx}
\usepackage{color}
\usepackage{caption}
\usepackage{subcaption}
\usepackage{float}
\usepackage{url}
\usepackage{authblk}

\newtheorem{theorem}{Theorem}[section]
\newtheorem{lemma}[theorem]{Lemma}

\theoremstyle{definition}

\numberwithin{equation}{section}
\title{\bf   Criteria for the Application of Double Exponential Transformation }
\author[1]{Arezoo  Khatibi}
\author[2]{Omid Khatibi\footnote{correspond email: \protect\url{md_khn@yahoo.com,a1449933@unet.univie.ac.at},Faculty of Mathematics,  University of Vienna,  Oskar morgernstern platz 1, 1090, Vienna, Austria.}} 
\affil[1]{University of Kashan,Kashan,Iran}
\affil[2]{ University of Vienna,Vienna,Austria}
\begin{document}
\maketitle
\begin{abstract}
The double exponential formula was introduced for calculating definite integrals   with singular point oscillation functions and Fourier-integrals. The double exponential transformation is not only useful for numerical computations but it is also used in different methods of Sinc theory. In this paper we use double exponential transformation for calculating particular improper integrals. By improving integral estimates having singular final points. By comparison between double exponential transformations and single exponential transformations it is proved that the error margin of double exponential transformations is smaller. Finally Fourier-integral and double exponential transformations are discussed. \\
 MATHEMATICS SUBJECT CLASSIFICATION: 65D30, 65D32.

\footnote

\end{abstract}

\section*{Introduction:}
The double exponential transformations (DE) is used for evaluation of integrals of an analytic function has end point singularity. This transformation improved frequently transformation like Fourier transformation. The fact that double exponential transformations error is smaller than single exponential transformation before this was seen as intuitive, but here we have proved it through theoretical analysis. Our numerical results are achieved in the maple software with different numerical results compared with Mori [1]. This article defines the criteria for using double exponential transformations. This technique can not be applied to sharp points. Sinc calculations are very accurate and complex. Regarding the above we have reached a numerical conclusion. In the following sections, we have expanded on the application of double exponential transformations worked in the Fourier integrals and also boundary value conditions. In this regard by incorporating new innovations in the lemma and theorems, we have improved on the above results. In [4] Ooura expended animt-type quadrature formula with the same asymptotic performance as the de-formula. Muhammad, Nurmuhammad and Mori the process by which a numerical  solution of integral equations by means of the Sinc collocation method based on the double exponential transformation has attracted considerable attention recently [5]. In a mathematically more rigorous manner, optimality of the double exponential formula-functional analysis approach is established by Sugihara [6]. Koshihara and Sugihara made a full study of A Numerical Solution for the Sturm–Liouville Type eigenvalue Problems employing the double exponential Transformation [8]. Ooura expended a double exponential formula for the Fourier transformation [10]. Stenger used DE Formula in Sinc approximations [7]. This article is organized as follows:\\
 In section 2 we state and prove a comparison between errors in DE and SE transformation. In section 3 we use DE transformation to compute several integrals and compare the result with that of Mori [1]. In section 4 we demonstrate an application of DE in solving a boundary value problem in section 5 we apply the DE transformation to Fourier integral operators and give a novel proof for the auxiliary theorem 2. 
\section{Conception of Double Exponential Transformation}
Let the following integral be 
\begin{equation}
\int_a^b\!f(x)\,\mathrm{d}x,
\end{equation}
such that the interval $(a,b)$ is infinite or half infinite and the function under integral is analytic  on $(a,b)$ and perhaps has singular  point  in $x=a$ or $x=b$ or both, now consider the  change  of  variables below [9]:
\begin{equation}\label{eq:varchange}
x=\phi(t),a=\phi(-\infty),b=\phi(+\infty),
\end{equation} 
 where $\phi$ is analytic on $(-\infty,+\infty)$ and 
 \begin{equation}\label{eq:defquadratur}
I=\int_a^b\!f(\phi(t))\phi\prime(t)\,\mathrm{d}t.
\end{equation}  
Hence, after the change of variable the integrand decays double exponential:
\begin{equation}\label{eq:doubleexponentialdecay}
\mid f(\phi(t))\phi\prime(t)\mid\approx e^{-c e^ {\mid t \mid}}, \mid t \mid\rightarrow\infty, c\textgreater 0. 
\end{equation}
By using   the trapezoidal formula with mesh size h on \eqref{eq:defquadratur} we have 
\begin{equation}\label{eq:trapez}
I_{h}=h\sum_{k=-\infty}^{+\infty}f(\phi(kh))\phi\prime(kh). 
\end{equation}
The above infinite summation is truncated from $k=〖-N〗^-$ to $k =N^+$ in computing \eqref{eq:trapez} we get
\begin{equation}
I_{h}^{(N)}=h\sum_{k=-N^-}^{N^+}f(\phi(kh))\phi\prime(kh),N=N^+ +N^- +1 
\end{equation}

 The N+ 1 is due to zero point, for example, $N=30$ we have $N^-=14, N^+ =15 $ finally 30=15+14+1.
What gives us the authority to truncate n numbers   from the   infinite summation? 
 In fact $e^{(-ce^{\mid t \mid} )}$ quickly   approaches zero.
  
For the integral over (-1, 1).
\begin{equation}
\int_ {-1}^ 1\!f(x)\,\mathrm{d}x,
\end{equation}
consider,
\begin{equation}
\phi(t) = \tanh (\frac{\pi}{2 }\sinh(t)),
\end{equation}
by substituting the transformation $ x= \phi(t) = \tanh (\pi/2 \sinh(t)) $  into \eqref{eq:trapez} we obtain
\begin{equation}
\phi(t)\prime= \tanh \prime (\pi/2\sinh(t))\pi/2\cosh(t)= \left(\frac {1}{\cosh^2(\pi/2\sinh(t))}\right)\pi/2\cosh(t).
\end{equation}
Consequently, by using the double exponential formula we will have:
\begin{equation}
I_{h}^{(N)}=h\sum_{k=-N^-}^{N^+}f(\tanh(\pi/2\sinh(kh))\frac{ \pi/2 \cosh(kh)}{\cosh^2 (\frac{\pi \sinh(kh)}{2})}. 
\end{equation}
\section{Numerical Experience and Comparison between D.E Approximation Error and S.E Transformation Approximation Error in Sinc Theory}

In this section we prove the main results of our paper. We compare two error estimates for the second-orders two point boundary problem exponential due to Horiuchi and Sugihara [2], which combine the double exponential transformation with the Sinc-Galerkin method, and prove the superiority of one estimate over the other.

As we laid out in [1], one considers
\begin{equation}
\til{y}^{''}(x)+\til{\mu}(x)\til{y}'(x)+\til{\nu}(x)\til{y}
(x)=\til{\sig}(x)\quad a<x<b,
\end{equation}
$\til{y}(a)=\til{y}(b)=0$,\\
by using the variable transformation\\
\begin{equation}
x=\var(t)\qquad a=\var(-\infty)\qquad b=\var(+\infty),
\end{equation}
since, together with the change of notation we have\\
\begin{equation}
y(t)=\til{y}(\var(t)),\label{eq:ytilda} 
\end{equation} 
transforms to 
 \begin{equation}
 {y}^{''}(x)+{\mu}(x){y}'(x)+{\nu}(x){y}
(x)={\sig}(x)\quad -\infty<t<+\infty,\label{eq:moadelemoshtaghdovominf}
  y(-\infty)=y(+\infty)=0,
 \end{equation}
 the Sinc-Galerkin method have ability to approximate the solution of \eqref{eq:moadelemoshtaghdovominf} by combination of Sinc function:
\begin{equation}
y_N (t)=\sum_{k=-n}^nw_k
S(k,h)(t), N=2n+1.
\end{equation}
 Where
 \begin{equation}
 S(k,h)(t)=\frac{\sin(\pi)/h(x-kh)}{\pi/h(x-kh)}, k=0, \pm1, \pm2,\dotsc 
 \end{equation}
 based on numerical result and theoretical proof the approximation error estimate by 
 \begin{equation}
 |y(t)-y_N (t) |\leq c'N^{\frac{5}{2}} exp(-c\sqrt{N}),
 \end{equation}
 when the true solution y(t) of the transform problem decays single exponentially leads to
 \begin{equation}
 |y(t) |\leq \alpha exp(-\beta|t|).
 \end{equation}
 Furthermore, the approximation error estimate by
 \begin{equation}
 |y(t)-y_N (t) |\leq c' N^2  \exp(\frac{-cN}{\log N}), 
 \end{equation}
 when the true solution y(t) of the transform problem decays double exponentially leads to
 \begin{equation}
 |y(t) |\leq \alpha \exp(-\beta  \exp(\ga|t|)).
 \end{equation}  
\begin{theorem}
In The above situation we prove that the approximation error arising from double exponential decay is less than  the approximation error arising from single exponential decay.
\end{theorem}
\begin{proof}
 The single exponential approximation error can be estimated by:
 \begin{equation}\label{eq:singleexponentialapproximateerror}
|y(t)-y_N (t) |\leq c'N^{\frac{5}{2}} exp(-c\sqrt{N}).
\end{equation}
\begin{equation}
y_N (t)=\sum_{k=-n}^nw_k
S(k,h)(t), N=2n+1.
\end{equation}
\\
$S(k,h)(t)$ is Sinc function.

The  downside   is  established  in  accordance  with  the  following  decay single exponential  $y(t)$,
\begin{equation}\label{eq:singleexponential}
|y(t) |\leq \alpha exp(-\beta|t|).
\end{equation}
The  double  exponential  approximation  error  can  be  estimate  by
\begin{equation}\label{eq:doubleexponentialapproximationerror}
|y(t)-y_N (t) |\leq c' N^2  \exp(\frac{-cN}{\log N}),
\end{equation}
\\
if  the  true  solution  of $y(t)$  the  transformed   problem  decay double 
exponentially  like[1]\\
\begin{equation}
|y(t) |\leq \alpha \exp(-\beta  \exp(\ga|t|)),
\end{equation}
we will  prove that  the  error in the  D.E  transformation   is  much less 
than the 
error in   the  S.E  transformation.

For  this  purpose we  need  lemma \ref{lemma1}.
\begin{lemma}\label{lemma1} There exists $N_0$ such that for all $N>N_0$, the inequality  
$c'N^2 e^{\frac{-cN}{\log N}}<c''N^{\frac{5}{2}} e^{-c\sqrt{N}}$ hold true.
\end{lemma}
\begin{proof}
\begin{equation}
\frac{c'}{c''}<N^{\frac{1}{2}}
e^{\frac{cN}{\log N} {-c\sqrt{N}}}.
\end{equation}

It is necessary  to be  prove: 
\begin{equation}
1<e^{\frac{cN}{\log N}-c\sqrt{N}}, 
\frac{c'}{c''}<N^{\frac{1}{2}},
\end{equation}
\begin{equation}
1<e^{\frac{cN}{\log N}-c\sqrt{N}}\LO 0<\frac{cN}{\log N}-c\sqrt{N},
\end{equation}
\begin{equation}
\frac{c'}{c''}<N^{\frac{1}{2}}\LO (\frac{c'}{c''})^2<N,
\end{equation}
\begin{equation}
0<\frac{cN}{\log N}-c\sqrt{N}\Lo cN>c\sqrt{N}\log N\LO \sqrt{N}>\log N
\overset{N=\log N}{\LO} e^{\frac{\log N}{2}}>\log N,
\end{equation}
according to  lemma \ref{lemma2} there exists $x_0$ such that for, 
\begin{equation}
 x>x_0, e^{\frac{x}{2}}>x,
\end{equation}
is now
\begin{equation} 
N>e^{x_0}\Lo \log N> x_0\overset{N=\log N}{\LO} e^{\frac{\log N}{2}}>\log 
N, 
\end{equation}
then put $N_0=max\{(\frac{c'}{c''})^2,e^{x_0}\}$
 and  the theorem  satisfied.
 \end{proof}
 For complete the proof we need other lemma. 

\begin{lemma}\label{lemma2}
 For each $a>0$  such 
that $t>2a$  we have $e^t>at$.
\end{lemma}
\begin{proof}
\begin{equation}
e^t=1+\frac{t}{1!}+\frac{t^2}{2!}+\frac{t^3}{3!}+\dots >1+t+\frac{t^2}{2},
\quad  t>0,
\end{equation}
for this we prove that,
\begin{equation}
  1+t+\frac{t^2}{2}>at,
  \end{equation}
this yields,
\begin{equation}    
 1+(1-a)t+\frac{t^2}{2}>0,
 \end{equation} 
 \begin{equation}
 t>\frac{a-1+\sqrt{(a-1)^2-2}}{1},
 \end{equation}
and
\begin{equation}
t>\frac{a-1-\sqrt{(a-1)^2-2}}{1},
\end{equation}
 the discriminant  of  the quadratic equation
 \begin{equation}   
1+(1-a)t+\frac{t^2}{2}=0,
\end{equation}
is
\begin{equation}  
\Del=(1-a)^2-2=a^2-2a-1=0,
\end{equation}                         
if  $\Del\geq 0$ we  put  $t_0$  equal  max root in quadratic equation  
$t_0=a-1+\sqrt{(a-1)^2-2}$.
Then  the statement  is satisfied. And  note  that
 \begin{equation}  
t_0<a+\sqrt{a^2}=2a. 
\end{equation}
If  $\Del<0$ that is  the roots are  complex  and the coefficient of  $t^2$ 
is positive then  for all of  t  inequality is true.
 Hence  we  can  set $t_0=0$  
and this completes the proof.

\end{proof}
\end{proof}

\section{NUMERICAL EXPERIENCE}
For  the  following  four  integral we get:
\begin{center}
\begin{tabular}{ccc}
\qquad\qquad\qquad\qquad DEFINT(Table 1)\\\hline
INTEGRAL & N & abs. error \\\hline
$I_1$ &25   &0.370582258 \\\hline
$I_2$ &  387 & 0.154551243\\\hline
$I_3$ & 387 & 0.021640875\\\hline
$I_4$ & 259 & 0.001118879\\\hline
\end{tabular}
\end{center}
In the  above table, the absolute 
  error  tolerance  is $10^{-8}$ where $N$  is the 
number of function  evaluations and abs. Error  is the actual absolute  error 
of  the result and  it gives an approximate value  which is correct up to 16 
significant digits.
\begin{align*}
&I_1=\int_0^1X^{\frac{-1}{4}} log \frac{1}{X}dx.\quad 
I_2=\int_0^1\frac{1}{16(X-\frac{\pi}{4})^2+\frac{1}{16}}dx. \qquad I_3=\int_0^\pi 
cos (64 sin X)dx.\\
& I_4=\int_0^1 exp(20(x-1))sin (256 x) dx.
\end{align*}

Our results are very different from [1] which was 
 Previously  defined. Compare 
quadrature or Horner method,  give at least $10^{-5}$ errors so  these method are
 better than the double exponential transformation method. 
\section{•Application
 of the Double-Exponential Transformation for the Sinc-Galerkin 
Method and the Second-Order Two-Point Boundary Problem}
Let
\begin{equation}
\til{y}^{''}(x)+\til{\mu}(x)\til{y}'(x)+\til{\nu}(x)\til{y}
(x)=\til{\sig}(x)\quad a<x<b,
\end{equation}
$\til{y}(a)=\til{y}(b)=0$,\\
by using the variable transformation\\
\begin{equation}
x=\var(t)\qquad a=\var(-\infty)\qquad b=\var(+\infty),
\end{equation}
since, together with the change of notation we have\\
\begin{equation}
y(t)=\til{y}(\var(t)),\label{eq:ytilda} 
\end{equation}
and
\begin{equation}
y'(t)=\var'(t)\til{y}'(\var(t)),\label{eq:derivative} 
\end{equation}
hence
\begin{equation}
y''(t)=\var''(t)\til{y}'(\var(t))+\var'(t)^2\til{y}''(\var(t)),\label{eq:seconderivative}
\end{equation}
in virtual point that \ref{eq:derivative} we obtain:
\begin{equation}
\til{y}'(\var(t))=\frac{1}{\var'(t)}y'(t). 
\end{equation}
In note that \ref{eq:seconderivative} we have:
\begin{equation}
\til{y}''(x)=-\til{\mu}(x)\til{y}'(t)-\til{\nu}(x)\til{y}(x)+\til{\sig}(x),
\end{equation}
\begin{equation}
y''(t)=\var''(t)\til{y}'(\var(t))+\var'(t)^2(-\til{\mu}(
\var(t))\til{y}'(\var(t))-\til{\nu}(\var(t))\til{y}(\var(t))+
\til{\sig}(\var(t))).\label{eq:seconderivative2}
\end{equation}
Our aim is change $\til{y}$ to y by using \ref{eq:ytilda},  \ref{eq:seconderivative2} we get follow 
equation:
\begin{equation}
y''(t)=\var''(t)\frac{1}{\var'(t)}y'(t)+\var'(t)^2(-
\til{\mu}(\var(t))\frac{1}{\var'(t)}y'(t)-\til{\nu}(\var(t))y(t)+
\til{\sig}(\var(t))),
\end{equation}
this however, leads to
\begin{equation}
y''(t)+(\var'(t)\til{\mu}(\var(t))-\frac{\var''(t)}{\var'(t
)})y'(t)+\var'(t)^2\til{\nu}(\var(t))y(t)=\var'(t)^2\til{\sig}(\var(t)), 
\end{equation}
\begin{equation}
\mu(t)=(\var'(t)\til{\mu}(\var(t))-\frac{\var''(t)}{\var'(t)}, 
\end{equation}
\begin{equation}
\nu(t)=\var'(t)^2\til{\nu}(\var(t)), 
\end{equation}
\begin{equation}
\sig(t)=\var'(t)^2\til{\sig}(\var(t)), 
\end{equation}
and observ that
\begin{equation}
y''(t)+\mu(t)y'(t)+\nu(t)y(t)=\sig(t).\label{eq:seconderivative3} 
\end{equation}
 By using the boundary conditions and \ref{eq:derivative} we have:
\begin{equation} 
\begin{cases}\til{y}(a)=\til{y}(\var(-\infty))=y(-\infty)=0\\
\til{y}(b)=\til{y}(\var(\infty))=y(\infty)=0
\end{cases}, 
\end{equation}
The Sink Galerkin  approximation  method provides 
 solutions to transform such as  the one in \ref{eq:seconderivative3} by combining it linearly with the Sinc function we have:
 \begin{equation}
y_N(t)=\sum_{k=-n}^n W_kS(k,h)(t),\quad N=2n+1. 
\end{equation}
Every method for trapezoidal numerical integration suggest one method in Sinc 
approximation. Horiuchi and Sugihara combine the Sinc-Galerkin method with the 
double exponential transformation for the second-order two-point boundary 
problem [2]. Mori and Nurmuhammad and Muhammad studied the performance of 
de-sinc method when used in the second order singularly perturbed boundary 
value problems [3].
\section{ Evaluation of Fourier-type integrals}
The double exponential is useful for several integrals for example Fourier 
integral
but it does not suitable for Fourier-type integrals slowly oscillatory function.
Fourier integral can compute from follow formula:
\begin{equation}
\hat{f}(w)=\int_{-\infty}^{+\infty } f(x)e^{ixw}dx.
\end{equation}
\begin{equation}
e^{ixw}=cos wx-i sinwx.\label{eq:4.2}
\end{equation}
Therefore we can show sin section with $I_s$ sign and cos section with 
$I_c$ sign then:
\begin{equation}
\begin{cases} Is=\int_0^\infty f_1(x) sin wx\\
 Ic=\int_0^\infty f_1(x) cos wx
\end{cases}\label{eq:isandic}. 
\end{equation}
we choose $\var(t)$ function applies to the following conditions.
\begin{equation}
\var(-\infty)=0 , \var(-\infty)=\infty,  
\end{equation}
if $t\lo -\infty$ 
 double exponentially
 \begin{equation}
\var'(t)\lo 0, \label{eq:phiderivativelimit}
\end{equation}      
if $t\lo \infty$ 
  double exponentially
\begin{equation}   
 \var(t)\lo t,\label{eq:limitapproacht}
 \end{equation}           
if $t\lo 0$ then  
\begin{equation}
\exists D|\var'(t)|\leq Dexp 
(-cexp|t|)|t|>0\LO \var'(t)\approx exp (-cexp|t|)\lo 
0 
\end{equation}
by using \ref{eq:limitapproacht}  we have:
\begin{equation}
|\frac{\var(t)}{t}|\approx exp (-cexp|t|),
\end{equation}
this means having the same exponentially growth ,on the other hand,
\begin{equation}
\var(t)-t\approx exp (-cexp|t|).
\end{equation}
Now by using change follow variable in transformation $I_S$, $I_C$ we have:\\
\begin{equation}
\begin{cases}
I_S: X=M\var(t)/w\\
I_C: X=M\var(t-\frac{\pi}{2M})/w\end{cases} (M=constant), 
\end{equation}
\begin{equation}
X=M\var(t)/w\Rightarrow dx=M\var'(t)/wdt .
\end{equation}
By using \ref{eq:isandic} we obtain:
\begin{equation}
I_S=\int_0^\infty f_1(x) sin wx=\int_0^\infty f_1(M\frac{\var(t)}{w})sin 
(M\var(t))\frac{M\var'(t)}{w}dt.
\end{equation}
\begin{theorem}Consider
\begin{equation}
\var(t)=\frac{t}{1-exp(-k sin ht)},
\end{equation}
if $t\lo -\infty$ 
 double exponentially $\var'(t)\lo 0$.
 \end{theorem}
\begin{proof}
\begin{equation}
\var'(t)=\frac{(1-e^{-ksin ht})-k cos hte^{-ksin ht} t}{(1-e^{-k sin 
ht})^2}, 
\end{equation}
if $ t\lo -\infty$
 then
\begin{equation}   
sinht=\frac{e^{+t}-e^{-t}}{2}\lo -\infty,
 \end{equation}
 and observe that
 \begin{equation}
\var'(t)=\frac{1}{1-e^{-k sin ht}}-\frac{kt cos hte^{-ksin ht}}{(1-e^{-ksin 
ht})^2},\label{eq:phiderivative}
\end{equation}
if  $t\lo -\infty $ then
\begin{equation}
 \frac{1}{1-e^{-ksin ht}}\lo 0,\label{eq:4.17} 
 \end{equation}
 since by \ref{eq:phiderivative} have
\begin{equation} 
= -\frac{ktcos hte^{-ksin ht}}{(1-e^{-ksin ht})^2}. \label{eq:4.18}
\end{equation}
However, if  the  second fraction of \ref{eq:phiderivative} is zero. Condition \ref{eq:phiderivativelimit} is 
satisfied.
\begin{lemma}\label{lemma4.2}
 If  $t\lo -\infty$ term 
$\frac{kt cos hte^{-ksin ht}}{(1-e^{-ksin ht})^2}$ 
 double exponentially approach to zero.
 \end{lemma}
\begin{proof}

 Consider
 \begin{equation} 
 A=\frac{kt cos ht}{1-e^{-k sin ht}}\times \frac{e^{-k sin 
ht}}{1-e^{-k sin ht}}, 
\end{equation}
\begin{equation}
|A|=|A_1|\times |A_2|, 
\end{equation}
such that
\begin{equation}
A_1=\frac{kt cos ht}{1-e^{-k sin ht}}, 
\end{equation}
\begin{equation}
A_2=\frac{e^{-k sin ht}}{1-e^{-k sin ht}}, 
\end{equation}
and
\begin{equation}
\lim_{t\lo -\infty} A_2=\lim_{x\lo 
\infty}\frac{x}{1-x}=\frac{\infty}{-\infty}=-1, 
\end{equation}
so that
\begin{equation}
\lim_{t\lo\infty}|A_2|=1. 
\end{equation}
For $t_0$ large enough, $|A_2|<2$,\\
numerator $A_1$ term $kt cosht$ exponentially approach to $\infty$ but 
denominator double
exponentially approach to $\infty$ then antecedent disable for neutralize 
denominator therefore term obtain  such that double exponentially approach to 
zero with $c$ less than $k$.\end{proof}
\paragraph{Claim:}
 For $c=\frac{k}{4}$
 we must prove 
 \begin{equation}
|A_1|<De^{Ce^{|t|}}. 
\end{equation}
$sinh t$ has constant  2 in its  denominator hence everything less 
than the value $\frac{k}{2}$
will work well but for the value 
 $k$  itself, the term equal to $\infty$.
\begin{lemma}There
 is $t_0$ such that for 
$t>t_0, |1-e^{-k sin ht}|>\frac{1}{2}e^{-k sin ht}.$\label{lemma4.3}
\end{lemma}
\begin{proof}
\begin{equation}
\lim_{t\lo-\infty}\frac{|1-e^{-k sin ht}|}{e^{-k sin ht}}=\lim_{x\lo\infty}
\frac{|1-x|}{x}=1, 
\end{equation}
since $1>\frac{1}{2}$
 therefore for $t>t_0$
 we have:
 \begin{equation}
\frac{|1-e^{-ksin ht|}}{e^{-ksin ht}}>\frac{1}{2},
\end{equation}
and
\begin{equation}
|A_1|<\frac{kt cos ht}{\frac{1}{2}e^{-ksin ht}}.\label{eq:4.28}
\end{equation}
It is trivial instead denominator $A_1$ in \ref{lemma4.3} from lemma \ref{lemma4.2}
We must prove 
\begin{equation}
|t|>t_0, |A_1|<De^{-\frac{k}{4}e^{|t|}},\label{eq:4.29}
\end{equation}
 for this aim we must prove
 \begin{equation}
|2kt cos hte^{ksin ht} e^{\frac{k}{4}e^{|t|}}|<D, 
\end{equation}
as a matter of fact we change denominator $A_1$
 \ref{eq:4.28} to numerator then implies
 \ref{eq:4.29} we get:
 \begin{equation}
t\lo-\infty, |t|=-t, 
\text{left side}=|2kt cosh te^{\frac{k}{2}}e^t e^{\frac{-k}{2}}e^{-t}
e^{\frac{k}{4}}e^{|t|}|<D, \label{eq:4.31}
\end{equation}
this yields
\begin{equation}
=|2kt cos ht e^{\frac{k}{2}e^t}e^{\frac{k}{4}e^{-t}}|<D,\label{eq:4.32} 
\end{equation}
in view of \ref{eq:4.32} we infer that
\begin{equation}
=|2kt cos ht e^{\frac{k}{2}e^t}e^{\frac{-k}{4}e^{-t}}|<D, 
\end{equation}
such that
\begin{equation}
=\lim_{t\lo-\infty}2kt cos hte^{\frac{k}{2}e^t}e^{\frac{-k}{4}e^{-t}}, 
\end{equation}
indeed,we have
\begin{equation}
\lim_{t\lo-\infty} e^{\frac{k}{2}e^t}=1, 
\end{equation}
and notice that we can write the equation
\begin{equation}
=2k\lim_{t\lo-\infty}cos ht e^{\frac{-k}{4}e^{-t}},\label{eq:4.36} 
\end{equation}
instead $cosht =\frac{e^t+e^{-t}}{2}$
 in \ref{eq:4.36}:
 \begin{equation}
=2k\lim_{t\lo -\infty}\frac{1}{2}(e^te^{\frac{-k}{4}e^{-t}}+ 
e^{-t\frac{k}{4}e^{-t}}), 
\end{equation}
we infer that
\begin{equation}
=k\lim_{t\lo-\infty}(e^{t\frac{k}{4}e^{-t}}+e^{-t\frac{k}{4}e^{-t}})=0.\label{eq:4.38}
\end{equation}
\end{proof}
\begin{lemma}
\begin{equation}
\lim_{t\lo-\infty}t-\frac{k}{4}e^{-t}=-\infty, 
\end{equation}
\begin{equation}
\lim_{t\lo-\infty}-t-\frac{k}{4}e^{-t}=-\infty.\label{lemma4.4}
\end{equation}
\end{lemma} 
\begin{proof}
\begin{equation}
\lim_{t\lo-\infty}t\times \frac{t-\frac{k}{4}e^{-t}}{t}=-\infty, 
\end{equation}
 for using L'H\^opital's rule we multiply $t$
 in denominator and numerator in fact we 
create $\frac{\infty}{\infty}$. 
\begin{equation}
\lim_{t\lo-\infty}t=-\infty,
\end{equation}
and observe that
\begin{equation}
\lim_{t\lo-\infty}\frac{t-\frac{k}{4}e^{-t}}{t}=\lim_{t\lo-\infty}
\frac{1+\frac{k}{4}e^{-t}}{1}=\infty, 
\end{equation}
we obtain
\begin{equation}
\lim_{t\lo -\infty}t\times \frac{-t-\frac{k}{4}e^{-t}}{t}=-\infty. 
\end{equation}

Hence \ref{eq:4.38} approach to zero therefore \ref{eq:4.31}, \ref{eq:4.32} for $t>t_0$ 
 is true.\\
We proved  \ref{eq:4.18} approach to zero then we will prove \ref{eq:4.17} approach to 
zero.\\
According the \ref{lemma4.2}  we have:\\
\begin{equation}
|\frac{1}{1-e^{-ksin ht}}|<\frac{1}{\frac{1}{2}e^{-ksin ht}}, 
\end{equation}
and
\begin{equation}
=\frac{2}{e^{-k(\frac{e^t-e^{-t}}{2})}}, 
\end{equation}
so that
\begin{equation}
=\frac{2}{e^{-k\frac{e^t}{2}}e^{k\frac{e^{-t}}{2}}}=
\frac{2e^{k\frac{e^t}{2}}}{e^{k\frac{e^{-t}}{2}}}, 
\end{equation}
this , however,leads to
\begin{equation}
\lim_{t\lo-\infty}2e^{k\frac{e^t}{2}}=1\Lo |t|>t_0, 2e^{k\frac{e^t}{2}}<1, 
\end{equation}

since $t<0$, 
absolute $t$ equal $-t$  gives :
\begin{equation}
|\frac{1}{1-e^{-ksin ht}}|<\frac{2}{e^{k\frac{e^{-t}}{2}}}=
2e^{-\frac{k}{2}e^{-t}}<De^{-Ce^{|t|}}. 
\end{equation}
And the proof is finished, hence the statement is satisfied.\end{proof}
\end{proof}

\paragraph{Conclusion.} In this paper the concept of the double exponential transformation was used to establish this method in spite of usefulness and being attractive it is not generally, we cannot use this method in all conditions especially in sharp point case and convergence rate slowly. The result in the oscillatory case is shown much better in $I_4$ than in the other cases $I_1, I_2, I_3$ (Table 1).  
\paragraph{Appendix  A:GALERKIN METHOD}
The Galerkin method, which approximates a solution of an equation of the form  
\begin{equation}
(1-\la k)f=g. \label{eq:4.50}
\end{equation}
defined on a Banach space $X$, is simple to describe. We start with a set of 
basis functions $\{\psi_k\}_{k=1}^n$, 
 and we assume that for each 
$f\in X$, 
we can determine a unique set of numbers $c_1,\dots ,c_n$ in $\BC$, 
such that the projection $P_{n:X\lo X}$ is 
defined, with \\
\begin{equation}
P_nf=\sum_{k=1}^n c_k\psi_k. 
\end{equation}
It means $p_n^2=p_n$. The Galerkin method then enables us to 
replace equation \ref{eq:4.50}
by the approximate one, namely \ref{eq:4.52},\\
\begin{equation}
(1-\la p_nk)f_n=p_ng,\label{eq:4.52} 
\end{equation}
why did we instead? Because system  \ref{eq:4.50} is infinite system but new system is 
system of n equations in n indeterminate.
\begin{equation}
(1-\la p_nk)\sum_{k=1}^nc_k\psi_k=p_ng, 
\end{equation}
and observe that
\begin{equation}
p_ng=\sum_{k=1}^n d_k\psi_k,
\end{equation} 
such that
\begin{equation}
(1-\la p_nk)\sum_{k=1}^n c_k\psi_k=\sum_{k=1}^nd_k\psi_k,
\end{equation} 
this, however, leads to
\begin{equation}
\sum_{k=1}^n c_k \psi_k-\la \sum_{k=1}^np_nk\psi_k=\sum_{k=1}^n 
d_k\psi_k,
\end{equation}
we write $p_n k\psi_k$ in terms of $\psi_j$,
\begin{equation}
p_nk\psi_k=\sum_{k=1}^n c_{kj}\psi_j,
\end{equation}
we infer that
\begin{equation}
\sum_{k=1}^nc_k\psi_k-\la \sum_{k=1}^n \sum_{j=1}^nc_{kj}\psi_j=\sum_{k=1}^n 
d_k\psi_k,
\end{equation}
we will be equal the coefficients $\psi_i$ on both sides,
\begin{equation}
c_i-\la \sum_{k=1}^n c_{ki}=d_i,\qquad i=1,\dots, n
\end{equation}
this process produces n equations, $n$ indeterminate.\\
In practice, we try to select projection as  $p_n$, to be comfortable working 
with them and $\|k-p_nk\|$
 being so small [7].
 


\begin{thebibliography}{9}
\bibitem{M. Mori , M. Sugihara2001}
  M. Mori , M. Sugihara,
  \emph{\:The Double Exponential Transformation in Numerical  Analysis },
  Journal Of  Computational And Applied  Mathematics Vol.127,
  (2001), 287-299. 
\bibitem{K. Horiouchi, M.Sugihara1999}
 K. Horiouchi, M.Sugihara,
  \emph{\:Sinc - Galerkin  Method  with  The Double Exponential Transformation for the Two Point Boundary Problems  },
   Technical Report, Dept. of Mathematical Engineering, University of Tokyo ,
  (1998), 99-105.
\bibitem{M.Mori, A.Nurmuhammad, M.Muhammad2009}
 M.Mori, A.Nurmuhammad, M.Muhammad,
  \emph{\:De-Sinc Method for Second Order Singularly Perturbed Boundary Value Problems  },
   japan j.Indust.Appl.Math, Vol. 26 ,
  (2009), 41-63.
\bibitem{T. Ooura2008}
 T. Ooura,
  \emph{\:Animt-Type Quadrature Formula with the Same Asymptotic Performance as the De-Formula  },
   Journal of Computational and Applied Mathematics, Vol. 213 ,
  (2008), 232-239
\bibitem{M. Muhammad, A. Nurmuhammad, M. Mori2005}
 M. Muhammad, A. Nurmuhammad, M. Mori,
  \emph{\:Numerical  Solution of Integral Equations by Means of the Sinc collocation Method Based on the Double Exponential Transformation   },
   Journal of Computational  and Applied Mathematics , Vol. 177 ,
  (2005), 269-286.
\bibitem{M. Sugihara1997}
 M. Sugihara,
  \emph{\:Optimality of The Double Exponential Formula –Functional Analysis Approach   },
   Numer Math, Vol. 75 ,
  (1997).
\bibitem{F. Stenger1993}
 F. Stenger,
  \emph{\:Numerical Methods Based on Sinc and Analytic Functions },
     
  (1993).
\bibitem{T. Koshihara, M. Sugihara1996}
 T. Koshihara, M. Sugihara,
  \emph{\:A Numerical Solution for The Sturm-Liouville 
Type Eigenvalue  Problems  Employing  The  Double  Exponential  Transformation  },
  Proceedings
 Of The 1996 Annual Meeting of The Japan Society for Industrial and 
Applied Mathematics, 1996, pp. 136-137(In Japanese).
\bibitem{O. Khatibi, A. Khatibi2017}
O. Khatibi, A. Khatibi,
 \emph{\:An Upper Bound Estimate and Stability for the
Global Error of Numerical Integration Using
Double Exponential Transformation  },
arXive, Preprints.
\bibitem{T. Ooura2005}
 T. Ooura,
  \emph{\:A Double Exponential Formula for Fourier Transform  },
   RIMS, Kyoto Univ., Vol. 41 ,
  (2005), 971-978.
 
\end{thebibliography}
\end{document}